\documentclass[british, 12pt, reqno]{amsart}
\usepackage[T1]{fontenc}
\usepackage[latin9]{inputenc}
\usepackage{amsmath, amsthm, amssymb, stmaryrd}
\usepackage{mathrsfs}
\usepackage{geometry}
\usepackage{mathtools}
\usepackage{amsmath}
\usepackage{amsthm}
\usepackage{amssymb}
\usepackage{xcolor}
\usepackage{dsfont}

\usepackage{cancel}
\usepackage{soul}
\usepackage{comment}
\usepackage{color}
\usepackage{amscd}
\usepackage{tabularx}
\usepackage{url}
\usepackage{eurosym}
\usepackage{braket}
\usepackage
{hyperref}
\usepackage{enumerate}
\usepackage{graphicx}

\makeatletter
\def\widebreve{\mathpalette\wide@breve}
\def\wide@breve#1#2{\sbox\z@{$#1#2$}%
     \mathop{\vbox{\m@th\ialign{##\crcr
\kern0.08em\brevefill#1{0.8\wd\z@}\crcr\noalign{\nointerlineskip}%
                    $\hss#1#2\hss$\crcr}}}\limits}
\def\brevefill#1#2{$\m@th\sbox\tw@{$#1($}%
  \hss\resizebox{#2}{\wd\tw@}{\rotatebox[origin=c]{90}{\upshape(}}\hss$}
\makeatletter

\makeatletter
\numberwithin{equation}{section}
\numberwithin{figure}{section}
\theoremstyle{plain}
\newtheorem{thm}{\protect\theoremname}[section]
\theoremstyle{plain}

\ifx\proof\undefined
\newenvironment{proof}[1][\protect\proofname]{\par
	\normalfont\topsep6\p@\@plus6\p@\relax
	\trivlist
	\itemindent\parindent
	\item[\hskip\labelsep\scshape #1]\ignorespaces
}{%
	\endtrivlist\@endpefalse
}
\providecommand{\proofname}{Proof}
\fi
\theoremstyle{remark}
\newtheorem{rem}[thm]{\protect\remarkname}
\theoremstyle{plain}
\newtheorem{lem}[thm]{\protect\lemmaname}

\makeatother
\DeclareMathOperator{\Bin}{Bin}
\usepackage{babel}
\providecommand{\corollaryname}{Corollary}
\providecommand{\lemmaname}{Lemma}
\providecommand{\remarkname}{Remark}
\providecommand{\theoremname}{Theorem}

\numberwithin{equation}{section}
\numberwithin{figure}{section}
\theoremstyle{plain}

\newtheorem{prop}[thm]{\protect Proposition}

\def\p{{\textcolor{red} {p}}}
\def\1{{\textcolor{red} {1}}}
\def\d1{{\textcolor{red} {d-1}}}

\def \bN {\mathbb N}
\def \bP {\mathbb P}
\def \bQ {\mathbb Q}
\def \QQ {\mathbb Q}
\def \bR {\mathbb R}

\def \bZ {\mathbb Z}

\def \bone {{\boldsymbol 1}}

\def \Gam {{\Gamma}}

\def \cI {\mathcal I}

\def \cL {\mathcal L}

\def \cT {\mathcal T}

\def \le {\leqslant}
\def \leq {\leqslant}
\def \ge {\geqslant}
\def \geq {\geqslant}

\DeclareMathOperator{\supp}{supp}

\def \d {{\mathrm{d}}}

\def \ds1 {\mathds{1}}

\def \alp {{\alpha}}

\def \gam {{\gamma}}
\def \del {{\delta}}
\def \eps {{\varepsilon}}

\def \lam {{\lambda}}

\newcommand{\NN}{\mathbb N}
\newcommand{\N}{\mathbb N}

\newcommand{\RR}{\mathbb R}
\newcommand{\x}{\mathbf x}

\newcommand{\y}{\mathbf y}
\newcommand{\ZZ}{\mathbb Z}

\makeatother

\usepackage{babel}
\providecommand{\lemmaname}{Lemma}

\providecommand{\theoremname}{Theorem}
\usepackage{xifthen,ifthen}
\makeatletter
\newcounter{@ToDo}
\newcommand{\todo@helper}[1]{%
	({\color{blue}TODO~\arabic{@ToDo}: {#1\@addpunct{.}}})%
}
\newcommand{\todo}[1]{\stepcounter{@ToDo}%
	\relax\ifmmode\text{\todo@helper{#1}}%
	\else\todo@helper{#1}\fi%
}
\newcounter{@cdo}
\newcommand{\cdo@helper}[1]{%
	({\color{red}CITE~\arabic{@cdo}: {#1\@addpunct{.}}})%
}
\newcommand{\cdo}[1]{\stepcounter{@cdo}%
	\relax\ifmmode\text{\cdo@helper{#1}}%
	\else\cdo@helper{#1}\fi%
}

\def \ds1 {\mathds{1}}

\def \alp {{\alpha}}

\def \gam {{\gamma}}
\def \del {{\delta}}
\def \eps {{\varepsilon}}
\def \epsilon {{\varepsilon}}

\def \lam {{\lambda}}

\DeclarePairedDelimiter{\abs}{\lvert}{\rvert}
\DeclarePairedDelimiter{\norm}{\lVert}{\rVert}

\DeclarePairedDelimiter{\setdelim}{\lbrace}{\rbrace}
\DeclarePairedDelimiter{\parens}{\lparen}{\rparen}
\DeclarePairedDelimiter{\brackets}{\lbrack}{\rbrack}

\begin{document}

\author{Sam Chow \and Manuel Hauke \and Andrew Pollington \and Felipe A.~Ram\'irez}
\address{Mathematics Institute, Zeeman Building, University of Warwick, Coventry CV4 7AL, United Kingdom}
\email{sam.chow@warwick.ac.uk}

\address{Department of Mathematical Sciences, NTNU Trondheim,
Sentralbygg 2, 7034 Trondheim, Norway}
\email{manuel.hauke@gmail.com}

\address{National Science Foundation, Arlington, VA 22230, USA}
\email{adpollin@nsf.gov}

\address{Department of Mathematics and Computer Science, Wesleyan University, Connecticut, United States}
\email{framirez@wesleyan.edu}

\title[General Duffin--Schaeffer-type counterexamples]{General Duffin--Schaeffer-type counterexamples in diophantine approximation}



\begin{abstract} 
Duffin and Schaeffer provided a famous counterexample to show that Khintchine's theorem fails without monotonicity assumption. Given any monotonically decreasing approximation function with divergent series, we construct Duffin--Schaeffer-type counterexamples by restricting the denominator. We also extend these constructions to the inhomogeneous setting. Our results resolve some natural questions arising from the works of Erd\H{o}s, Vaaler, and Yu.
\end{abstract}

\maketitle

\section{Introduction}

For $\psi: \bN \to [0,\infty)$, the set of $\psi$-well approximable numbers in $[0,1]$ is
\[
W(\psi) = \{ \alp \in [0,1]: \exists^\infty q \in \bN \quad \| q \alp \| < \psi(q) \},
\]
where $\| x \| = \min \{ |x - a|: a \in \bZ \}$. Khintchine's theorem \cite[Theorem 2.3]{BRV2016}, the cornerstone of metric diophantine approximation, asserts that if $\psi$ is non-increasing then
\begin{equation}
\label{KhintchineConclusion}
\lam(W(\psi)) = \begin{cases}
0, &\text{if } \displaystyle \sum_{q=1}^\infty \psi(q) < \infty \\ \\
1, &\text{if } \displaystyle \sum_{q=1}^\infty \psi(q) = \infty.
\end{cases}
\end{equation}
Here and throughout, we write $\lam$ for Lebesgue measure on $\bR$. Duffin and Schaeffer~\cite{DS1941} famously constructed a counterexample to show that monotonicity cannot be relaxed in general. In the same article, they formulated the Duffin--Schaeffer conjecture on non-monotonic approximation with reduced fractions. Building upon the important steps made by Erd\H{o}s \cite{Erd1970}, Vaaler \cite{Vaa1978} and Pollington--Vaughan~\cite{PV1990}, the conjecture was ingeniously resolved by Koukoulopoulos and Maynard \cite{KM2020}. 

\begin{thm} [Koukoulopoulos and Maynard]
\label{KMthm}
Let $\psi: \bN \to [0, \infty)$ and
\[
W'(\psi) = \{ \alp \in [0,1]: \exists^\infty (q,a) \in \bN \times \bZ \qquad | q \alp  - a| < \psi(q), \quad (a,q) = 1 \}.
\]
Then
\begin{equation}
\label{DSconclusion}
\lam(W'(\psi)) = \begin{cases}
0, &\text{if } \displaystyle \sum_{q=1}^\infty \frac{\varphi(q)\psi(q)}{q} < \infty \\ \\
1, &\text{if } \displaystyle \sum_{q=1}^\infty \frac{\varphi(q)\psi(q)}{q} = \infty.
\end{cases}
\end{equation}
\end{thm}

\subsection{Main results}

Our work is principally motivated by the aforementioned results of Erd\H{o}s and Vaaler. The latter asserts that
\[
\psi(q) \ll 1/q \quad \forall q
\: \:
\Longrightarrow
\: \:
\text{\eqref{DSconclusion}}.
\]
This raises the question of whether or not 
\[
\psi(q) \ll 1/q \quad \forall q
\: \:
\Longrightarrow
\: \:
\text{\eqref{KhintchineConclusion}}.
\]
We answer this question negatively,
showing \emph{a fortiori} that if $f: \bN \to [0,\infty)$ is \textit{any} non-increasing function such that 
\begin{equation}
\label{fdiverge}
\lim_{q \to \infty} f(q) = 0, \qquad
\sum_{q=1}^\infty f(q) = \infty,
\end{equation} 
then
\[
\psi(q) \le f(q) \quad \forall q \: \:
\not \Longrightarrow
\: \:
\text{\eqref{KhintchineConclusion}}.
\]

\begin{thm}
\label{thmA}
Let $f: \bN \to [0,\infty)$ be a non-increasing function satisfying \eqref{fdiverge}. Then there exists $S \subseteq \bN$ such that
\[
\sum_{q \in S} f(q) = \infty,
\qquad
\lam(W(f \bone_S)) = 0.
\]
\end{thm}

Though it may at this stage be clear to a handful of experts, we wish to point out that our constructions can be modified to obtain Duffin--Schaeffer-type counterexamples that are monotonic on a support that has full density at infinitely many scales. Writing
\[
\bar \del (S) = \limsup_{N \to \infty} \frac{|S\cap[1,N]|}{N}
\]
for the upper density of $S \subseteq \bN$, we establish the following refinement of Theorem~\ref{thmA}. 

\begin{thm}
\label{thmC}
Let $f: \bN \to [0,\infty)$ be a non-increasing function satisfying \eqref{fdiverge}. Then there exists $\psi: \bN \to [0,\infty)$, non-increasing on its support, such that  
\begin{equation}
\label{eq:thmC}
\psi\leq f,
\qquad
\sum_{q \in \NN} \psi(q) = \infty,
\qquad
\lam(W(\psi)) = 0,
\qquad
\bar \del(\supp\psi) = 1.
\end{equation}
\end{thm}

\begin{rem}
\label{rem:but}
Note that, in Theorem \ref{thmC}, it cannot be guaranteed that $\psi$ will take the form $\bone_S f$ for some $S\subset\NN$ with full upper density. For example, if $f(n) \gg 1/\varphi(n)$, then Theorem \ref{KMthm} precludes the existence of such a set $S$. We explain this in \S  \ref{sec:upperdensityproof}.

Note also that ``upper density'' in Theorem \ref{thmC} cannot be replaced by ``lower density'', since the Duffin--Schaeffer theorem \cite[Theorem III]{DS1941} --- not to be confused with the Duffin--Schaeffer conjecture --- implies that if
\[
\liminf_{N \to \infty} \frac{|S\cap[1,N]|}{N} > 0
\]
then
\[
\sum_{q \in S}\psi(q) = \infty \implies \lam(W(\psi)) = 0.
\]
\end{rem}

\subsection{Inhomogeneous results}

There has been considerable interest around inhomogeneous, non-monotonic approximation in recent years \cite{AR,BHV, CT2024, H25, HR, K25, Ram2017, Yu2021}. For $\psi: \bN \to [0,\infty)$ and $\gam \in \bR$, define
\[
W(\psi, \gam) =
\{ \alp \in [0,1]: \exists^\infty q \in \bN \quad \| q \alp - \gam \| < \psi(q) \}.
\]
The second motivation for our work comes from the aforementioned article of Yu~\cite{Yu2021}. A \emph{Liouville number} is $\gam \in \bR \setminus \bQ$ such that if $w \in \bR$ then
\[
\| q \gam \| < q^{-w}
\]
holds for infinitely many $q \in \bN$. These form a set $\cL$. Yu introduced the set $\cT$ of \textit{tamely Liouville numbers}, which satisfies
\[
\bR \setminus (\bQ \cup \cL)
\subsetneq \cT \subsetneq \bR \setminus \bQ,
\]
and showed in \cite[Theorem 1.3]{Yu2021} that
\begin{equation} 
\label{YuThm}
\psi(q) \ll \frac1{q(\log \log q)^2} \: \: \Longrightarrow \: \: 
\left(
\sum_{q=1}^\infty \psi(q) = \infty
\: \: \Longrightarrow \: \:
\lam(W(\psi,\gam)) = 1 \: \: \forall \gam \in \cT \right).
\end{equation}
This raises the question of whether the tame Liouville assumption can be replaced by the weaker assumption of irrationality. We answer this question negatively,
showing \emph{a fortiori} that if $f: \bN \to [0,\infty)$ is a non-increasing function satisfying \eqref{fdiverge}, then
\[
\psi(q) \le f(q) \quad \forall q
\: \: \not \Longrightarrow \: \: 
\left(
\sum_{q=1}^\infty \psi(q) = \infty
\: \: \Longrightarrow \: \:
\lam(W(\psi,\gam)) = 1 \: \: \forall \gam \in \bR \setminus \bQ \right).
\]

\begin{thm}
\label{thmB}
Let $f: \bN \to [0,\infty)$ be a non-increasing function satisfying
\eqref{fdiverge}.
Then there exists a set $S \subseteq \NN$ and an uncountable set $\Gamma \subseteq \RR$ containing $\QQ$ such that if $\gam \in \Gam$ then
\begin{equation}
\label{eq:thmB}
\sum_{q \in S} f(q) = \infty, \qquad
\lam(W(\bone_S f, \gam)) = 0.
\end{equation}
\end{thm}

Far from being tamely Liouville in the sense of Yu, the set $\Gamma$ in Theorem~\ref{thmB} consists of real numbers that are ``wildly'' Liouville in a sense that depends on the function $f$. In fact, to each $f$ we are able to associate a notion of wildly Liouville that is sufficient for the existence of a set $S = S_\gamma$ such that~\eqref{eq:thmB} holds (see Theorem~\ref{thm:generalhowliouville}). As a consequence, we have the following theorem for functions of the form
\begin{equation}
\label{special}
f_k(q)= \frac{1}{q(\log q) (\log\log q) \cdots \log^{(k)}q},
\end{equation}
where $\log^{(k)}$ denotes the $k^{\mathrm{th}}$ iterate of 
$x \mapsto \max \{1, \log x \}$.

\begin{thm}
\label{thm:howliouville}
Let $k\geq 0$ be an integer, and let $\gamma\in\RR$. If
\begin{equation*}
\liminf_{q \to \infty} \lVert q \gamma \rVert \exp^{(k+3)}(q^{7000}) < \infty,
\end{equation*}
where $\exp^{(k)}$ denotes the $k^{\mathrm{th}}$ iterate of $\exp$, then there exists $S \subseteq \N$ such that
\[
\sum_{q \in S}f_k(q) = \infty, \qquad \lambda(W(\bone_Sf_k,\gamma)) =
0.
\]
\end{thm}

In particular, Yu's result cannot be extended to real numbers $\gamma$ for which
\[
\liminf_{q \to \infty} \norm{q\gamma}\exp^{(4)}(q^{7000}) < \infty.
\]
We remark that there is an enormous gap between this wild Liouville notion and Yu's tame Liouville notion. While one can most probably sharpen the estimates applied in this article, it does not seem that the gap can be closed simply by optimising the proof of Theorem~\ref{thm:howliouville}.

\

Our constructions have the novel feature of monotonicity on the support. In previous Duffin--Schaeffer-type constructions, the approximating function $\psi$ has the weaker property that $q \mapsto \psi(q)/q$ is monotonic.

\

We conclude our introductory discussion with an open problem, to prove or disprove the following quantitative refinement of \eqref{YuThm}, that
\[
\psi(q) \ll 1/q \: \: \Longrightarrow \: \: 
\left(
\sum_{q=1}^\infty \psi(q) = \infty
\: \: \Longrightarrow \: \:
\lam(W(\psi,\gam)) = 1 \: \: \forall \gam \in \cT \right).
\]

\subsection{Methods}

Let us describe the  proof of Theorem \ref{thmA} in coarse terms. First of all, we require that
\[
\lam \left( \bigcup_{q \in S'} A_q \right) = o \left(
\sum_{q \in S'} \lam(A_q)
\right),
\]
where $S$ is a disjoint union of sets of the form $S'$. Secondly, we require the series $\sum_{q \in S} f(q)$ to diverge very slowly. To these ends, we take
\[
S' = X \cdot Y := \{ xy: x \in X, \: y \in Y \}.
\]
The set $Y$ is highly multiplicatively structured. We use some probability theory to show that a generic pair of elements of $Y$ has a large common divisor. This enables us to deal with the first requirement. The set $X$ is then chosen, consisting of much larger elements, so as to carefully secure the second requirement.

For the inhomogeneous theory, say Theorem \ref{thmB}, our shift $\gam$ is very close to some rational number with denominator $b$, in fact a sequence of these. The key extra ingredient, Lemma~\ref{lem:simple}, asserts that if $\| b \gam \| \le b \eps$ then $A_q^\gam(\eps) \subseteq A_{bq}(2b \eps)$. This essentially reduces the problem to the homogeneous setting.

\subsection*{Notation}

For $\eps>0, \gamma\in\RR$, and $q\in\NN$, denote
\begin{equation*}
    A_q^\gamma(\eps) = \setdelim*{\alpha\in[0,1]: \norm{q\alpha - \gamma} < \eps}. 
\end{equation*}
For $f:\NN\to\RR_+$, we denote $A_q^\gamma(f) = A_q^\gamma(f(q))$, and when it is clear from the context which function we are working with, we write $A_q^\gamma$ for this set. Thus,
\begin{equation*}
    W(f, \gamma) = \limsup_{q\to\infty} A_q^\gamma(f).
\end{equation*}
When there is no superscript $\gamma$, the notation should be interpreted with $\gamma=0$.

We write $\bP = \{ p_1 < p_2 < \ldots \}$ for the set of all primes. For $m \in \bN$, we abbreviate $\bZ_m = \bZ/m\bZ$.

\subsection*{Funding} AP was supported by the National Science Foundation.

\subsection*{Rights}

For the purpose of open access, SC has applied a Creative Commons Attribution (CC-BY) licence to any Author Accepted Manuscript version arising from this submission.

\section{Homogeneous counterexamples}

In this section, we prove Theorems \ref{thmA} and \ref{thmC}.

\subsection{Proof of Theorem \ref{thmA}}

\subsubsection{Construction of counterexample blocks}

The statement of the Theorem \ref{thmA} will follow straightforwardly from the following proposition.

\begin{prop}
\label{counterexample_blocks}
Let $f: \N \to [0,\infty)$ be a non-increasing function satisfying \eqref{fdiverge}.
Then, for any $\varepsilon, M > 0$, there exists a finite set $S \subset \N \cap (M,\infty)$ such that 
\begin{equation}
\label{eq_counter_block}
\sum_{q \in S}\lambda(A_q) \in [1,2],
\qquad \lambda \left(\bigcup_{q \in S} A_q \right) < \varepsilon.\end{equation}
\end{prop}

The proof of Proposition 
\ref{counterexample_blocks} will be quite lengthy. Assuming it for now to be given, Theorem \ref{thmA} follows from a standard procedure which we include for completeness.

\begin{proof}
[Proof of Theorem \ref{thmA}, assuming Proposition \ref{counterexample_blocks}]
We construct pairwise disjoint sets $S_1, S_2, \ldots$ recursively using Proposition \ref{counterexample_blocks}. First, we choose $M = 1$ and $\eps = 1/2$ to obtain $S_1$ satisfying \eqref{eq_counter_block} with $S_1$ in place of $S$. Having constructed $S_1,\ldots,S_{k-1}$ in this way, we choose $M_k = \max S_{k-1}$ and apply Proposition \ref{counterexample_blocks} with $\eps = 2^{-k}$. Finally we set $S = \bigcup_{k \in \N}S_k$. We see from
\eqref{eq_counter_block} that
\[\sum_{q \in S}\lambda(A_q) = \sum_{k \in \N}\sum_{q \in S_k}\lambda(A_q) = \infty.\]
On the other hand,
\[
\lambda(W \bone_S) 
\leq \lim_{K \to \infty}\sum_{\ell \geq K}\lambda\left(\bigcup_{q \in S_{\ell}}A_q\right)
\leq \lim_{K \to \infty}\sum_{\ell \geq K} 2^{-\ell} = 0.
\]
\end{proof}

We now proceed in stages to prove
 Proposition \ref{counterexample_blocks}.
For $\cI \subseteq \N$, let 
\[
Y(\cI) = \left\{\prod_{i \in \cI} p_{2i-a_i}: a_i \in \{0,1\}\right\}, \quad P(\cI) = \prod_{i \in \cI} p_{2i}p_{2i-1},
\]
and define the function
\[
g_{\cI}(z) = \prod_{\substack{i \in \cI\\p_{2i}\mid z}} \left(1 - \frac{1}{p_{2i}}\right).
\]
With this notation, we make the following observation.

\begin{lem}
\label{lem_overlap_bound}
Let $x \in \bN$ with $(x,P(\cI)) = 1$, and let $f: \bN \to [0,\infty)$ be non-increasing. Then
\begin{equation}
\label{eq_overlap_bound}
\lambda\left(\bigcup_{y \in Y(\cI)}A_{xy}\right) \leq 2\sum_{y \in Y(\cI)}g_{\cI}(y)f(xy).
\end{equation}
\end{lem}

\begin{proof}
Clearly,
\[\lambda\left(\bigcup_{y \in Y(\cI)}A_{xy}\right) = \sum_{y \in Y(I)}\lambda\left(A_{xy}\setminus \bigcup_{\substack{y' \in Y(\cI)\\y' < y}}A_{xy'}\right),\]
and we will bound the right-hand side from above term-wise. Observe that
\[
A_{xy} = \bigcup_{a \in \mathbb{Z}_{xy}}\left[\frac{a}{xy} - \frac{f(xy)}{xy},\frac{a}{xy} + \frac{f(xy)}{xy}\right]
\pmod{1}.
\]
Since $f$ is non-increasing, so is $f(y)/y$. Thus, if $\frac{a}{xy} = \frac{b}{xy'}$ for some $b \in \ZZ$ and some $y' < y$, then
$\left[\frac{a}{xy} - \frac{f(xy)}{xy},\frac{a}{xy} + \frac{f(xy)}{xy}\right] \pmod{1} \subseteq A_{xy'}$. Hence,
\begin{align}
\notag
&\sum_{y \in Y(\cI)}\lambda\left(A_{xy}\setminus \bigcup_{\substack{y' \in Y(\cI)\\y' < y}}A_{xy'}\right) \\
\notag
&\leq 2\frac{f(xy)}{xy}\#\left\{a \in \mathbb{Z}_{xy}: \not\exists y'\in Y(\cI), y' < y, b \in \ZZ: \frac{a}{xy} = \frac{b}{xy'}\right\}
\\
\label{reduced_to_comb}
&= 2\frac{f(xy)}{y}\#\left\{a \in \mathbb{Z}_{y}: \not\exists y'\in Y(\cI), y' < y, b \in \ZZ: \frac{a}{y} = \frac{b}{y'}\right\}.
\end{align}

For $y \in Y(\cI)$, we decompose $y = y_ey_o$, where 
\[
y_e = \gcd\left(\prod_{i \in \cI}p_{2i}, y\right), \quad y_o = \gcd\left(\prod_{i \in \cI}p_{2i-1}, y\right).
\]
We claim that if $a \in \ZZ_y$ contributes to the count in \eqref{reduced_to_comb}, then $(a,y_e) = 1$. Indeed, suppose $(a,y_e) > 1$. Then there exists $i \in \cI$ such that $p_{2i} \mid a$ and $p_{2i} \mid y$.
Defining the ``conjugate'' $c_i(y) = y \frac{p_{2i-1}}{p_{2i}} \in Y(\cI)$, we observe that $c_i(y) < y$. Since $y \mid p_{2i} c_i(y) \mid a c_i(y)$, there exists $b \in \ZZ$ such that $\frac{a}{y} = \frac{b}{c_i(y)}$. This confirms the claim. Therefore
\[
\#\left\{a \in \mathbb{Z}_{y}: \not\exists y'\in Y(\cI), y' < y, b \in \ZZ: \frac{a}{y} = \frac{b}{y'}\right\}
\leq y_o\varphi(y_e) = yg_{\cI}(y).
\]
Substituting this into \eqref{reduced_to_comb} completes the proof.
\end{proof}

\subsubsection{Construction of well-behaving $\cI$}

\begin{lem}
\label{lem_good_I}
For $\delta > 0$ and $\cI \subseteq \N$, let 
\[
Y(\cI,\delta) = \{y \in Y: g_{\cI}(y) \geq \delta\}, \quad y_{\min} = \min_{y \in Y(\cI)}y.
\]
Then for any $\delta > 0$, there exists $\cI \subseteq \N$ such that
\begin{equation}
\label{eq_good_I}
\frac{\#Y(\cI,\delta)}
{\sum\limits_{y \in Y(\cI)}\frac{y_{\min}}{y}} < \delta.
\end{equation}
\end{lem}

\begin{proof}
For each $j \in \N$, we define
\[
I_{j} = \left\{i \in \N: p_{2i-1},p_{2i} \in \mathbb{P} \cap [e^j,e^{j+1}), \: p_{2i} - p_{2i-1} \leq 40j \right\},
\]
and for $K \in \N$, we set
\[
\cI_K = \bigcup_{5 \le j \leq K} I_j.
\]
For any fixed $\delta> 0$, we will prove that if $K$ is sufficiently large then $\cI = \cI_K$ satisfies \eqref{eq_good_I}.

To do so, we start by finding a lower bound for the denominator in \eqref{eq_good_I}. Let us abbreviate $Y = Y(\cI)$. Using the condition on the gaps, we see that for $K$ sufficiently large,
\begin{align}
\notag
\sum_{y \in Y} \frac{y_{\min}}{y} &= \prod_{i \in \cI}\left(1 + \frac{p_{2i-1}}{p_{2i}}\right) 
= 2^{\#\cI} \prod_{i \in \cI}\left(1 - \frac{p_{2i} - p_{2i-1}}{2p_{2i}}\right) \\
\notag
&\geq \# Y \prod_{p \leq e^{K+1}}\left(1 - \frac{20 \log p}{p}\right)
\geq \# Y \exp\left(-40 \sum_{p \leq e^{K+1}} \frac{\log p}{p}\right) \\
\label{eq_harmonic_avg}
& \geq \# Y \exp(-50K),
\end{align}
where we used Mertens' first theorem in the last line.

Next, we find an upper bound for $\#Y$.
Denoting
\[
X_j(y) = \# \{ i \in I_j: p_{2i} \mid y \},
\]
we see that for $K$ sufficiently large,
\begin{equation*}
g_{\cI}(y) \leq \prod_{j \leq K}\left(1 - \frac{1}{e^{j+1}}\right)^{X_j(y)}
\leq \exp\left(-\sum_{(\log K)^2 \leq j \leq K} \frac{X_j(y)}{e^{j+1}}\right).
\end{equation*}
If $X_j(y) \geq \frac{e^j}{500j}$ for all $(\log K)^2 \leq j \leq K$, then this gives 
\[
g_{\cI_K}(y) \leq K^{-1/2000} < \delta,
\]
provided that $K$ is sufficiently large. The upshot is that
\begin{equation}
\label{eq_inclusion}
Y(\cI,\delta) \subseteq 
\left\{ y \in Y: 
\exists j \in [(\log K)^2, K]: X_j(y)
\leq \frac{e^j}{500j}\right\}.
\end{equation}
We will now be bound the cardinality of this set, from above.

Choosing the uniform probability measure on $Y$, i.e. $\mathbb{P}(y) = \frac{1}{\lvert Y \rvert}, y \in Y$, we can interpret $(X_j)_{j = 1}^k$ as random variables. We can see immediately that for all $1 \leq j \leq K$, we have 
$X_j \stackrel{d}{=} \Bin(\#I_j,\tfrac{1}{2})$ where $\Bin(n,p)$ denotes the binomial distribution. An application of Hoeffding's inequality thus shows that
\[\mathbb{P}\left[X_j(y) \leq \tfrac{\#I_j}{10}\right] = \frac{\#\left\{y \in Y: X_j(y) \leq \tfrac{\#I_j}{10}\right\}}{\#Y}
\leq \exp\left(-\tfrac{\#I_j}{4}\right).
\]
Next, we bound for $\#I_j$ from below. By the prime number theorem, if $j$ sufficiently large then
\[
\#\{p \in [e^j,e^{j+1}): p \in \mathbb{P} \} \geq \frac{e^j}{5j}.
\]
This implies that the average prime gap in $[e^j,e^{j+1})$ is at most $10j$, and so the average gap arising between the primes $p_{2i}$ in this range is at most $20j$. Thus Markov's inequality implies that, for sufficiently large $j$,
\begin{equation*}
\label{size_Ij}\#I_j \geq \frac{e^j}{50j}.
\end{equation*}

For $K$ sufficiently large, the union bound now yields
\begin{align*}
&\frac{\#\left\{y \in Y: 
\exists j \in [(\log K)^2, K]: X_j(y)
\leq \frac{e^j}{500j}
\right\}}{\#Y} \\
&\leq \sum_{(\log K)^2 \leq j \leq K} \mathbb{P}\left[X_j(y) \leq \frac{\# I_j}{10}\right]
\leq \sum_{(\log K)^2 \leq j \leq K} \exp\left(-\frac{\# I_j}{4}\right)
\\&\leq \sum_{j \geq (\log K)^2} \exp\left(-\frac{e^j}{200}\right) \ll \exp(-100K).
\end{align*}
Combining this with \eqref{eq_harmonic_avg} and \eqref{eq_inclusion} completes the proof.
\end{proof}

\subsubsection{Completion of the proof of Proposition \ref{counterexample_blocks}}
\label{subsec_completion}

We assume, as we may, that the constant $\eps > 0$ is sufficiently small.  By Lemma \ref{lem_good_I}, there exists $\cI$ such that \eqref{eq_good_I} is satisfied, for some $\delta = \delta_{\varepsilon}$ to be specified later. Denote 
\begin{equation}
\label{denote}
Y = Y(\cI), \qquad P = P(\cI) = \prod_{i \in \cI} p_{2i}p_{2i-1},
\end{equation}
as well as
\[
y_{\min} = \min Y, \qquad
y_{\max} =  \max Y.
\]
We construct the set $S$ in the following way.

\begin{enumerate}
\item \label{step1}
We choose $x_{\min}' \geq M$ large enough for $f(x_{\min}'y_{\min})$ to be arbitrary small, even if multiplied by whatever term we will need depending on $Y$. We denote such a quantity by $o_{Y}(1)$.
\item \label{step2}
Next, we choose $B \in \bN$, coprime to $P$, such that $BP + 1 \ge x'_{\min}$.
Then we define 
$
x_{\min} = BP + 1,
$
noting that 
\begin{equation}
\label{control_xmin}
x_{\min}' \leq x_{\min} \leq 2BP.
\end{equation}
\item \label{step3}
We claim that there exists $x_{\max} \in \N$, depending on the quantities defined above, such that
\begin{equation}
\label{const_mass}
\sum_{y \in Y} \sum_{\substack{x \in [x_{\min},x_{\max}]\\x \equiv 1 \pmod {BP}}} f(xy) \in [1,2].
\end{equation}
Indeed, let $h(x) = \sum_{y \in Y}f(xy)$. Then  
\[
\sum_{x=1}^\infty
h(x) \geq \sum_{x \in \N} f(y_{\min}x) = \infty
\]
and, since $h$ inherits monotonicity from $f$, 
\[
\sum_{\substack{x \in \N\\x \equiv 1 \pmod {BP}}}h(x) = \infty.
\]
Since $f(q) \to 0$ as $q \to \infty$, we have 
\[
h(x) = o_Y(1) < 1 \qquad (x \geq x_{\min}).
\]
This confirms that a positive integer $x_{\max}$ satisfying \eqref{const_mass} can be found. We may clearly take $x_{\max} \equiv 1 \pmod{BP}$.
\item We set 
\[
X = \{ x \in [x_{\min},x_{\max}]: x \equiv 1 \pmod {BP}\}
\]
and 
\[
S = X \cdot Y := \{ xy: x \in X, \: y \in Y \}.
\]
\end{enumerate}

Having constructed $S$, we apply Lemma \ref{lem_overlap_bound} to obtain
\[
\lambda\left(\bigcup_{q \in S} A_{q}\right) \leq 
\sum_{x \in X} \lambda \left(\bigcup_{y \in Y} A_{xy} \right)
\leq 2\sum_{x \in X} \sum_{y \in Y} g_{\cI}(y) f(xy)
= 2\sum_{y \in Y} g_{\cI}(y)t(y),
\]
where $t(y) = \sum_{x \in X}f(xy)$. Since $f$ is non-increasing,
\begin{equation*}
\label{eq_AP_sandwich}
\frac{1}{yBP}\sum_{0 \leq k \leq yBP-1}f(xy+k)\leq f(xy) \leq \frac{1}{yBP}\sum_{0 \leq k \leq yBP-1}f(xy-k).
\end{equation*}
Consequently,
\[
\begin{split}
t(y) &\leq f(x_{\min}y) + \frac{1}{yBP} \sum_{\substack{ x_{\min} < x \le x_{\max} \\ x \equiv 1 \pmod{BP}}} \sum_{0 \leq k \leq yBP-1}f(xy-k)\\&= f(x_{\min}y) + \frac{1}{yBP} \sum_{z = yx_{\min}+1}^{yx_{\max}} f(z) 
\end{split}
\]
and
\[
t(y) \geq \frac{1}{yBP} 
\sum_{\substack{ x_{\min} \le x \le x_{\max} \\ x \equiv 1 \pmod{BP}}} 
\sum_{0 \leq k \leq yBP-1} f(xy-k)
= \frac{1}{yBP} \sum_{z = yx_{\min}+1}^{yx_{\max}} f(z).
\]
Thus, writing
\[
F(y) = \frac{1}{BP} \sum_{z = yx_{\min}+1}^{yx_{\max}} f(z),
\]
we conclude that
\begin{equation*}
\label{eq_t=F}
t(y) = \frac{F(y)}{y} + O(f(x_{\min}y)) = \frac{F(y)}{y} + o_{Y}(1).
\end{equation*}
For $x_{\min}'$ sufficiently large, it follows that
\begin{equation*}
\lambda\left(\bigcup_{q \in S}A_{q}\right) \leq 2\left(\sum_{y \in Y}\frac{g_{\cI}(y)}{y}F(y)\right) + \delta.
\end{equation*}

Since $f \ge 0$ is non-increasing, so too is $F(y)/y$. Using the notation 
\[
Y(\cI,\delta) = \left\{y \in Y: g_{\cI}(y) \geq \delta\right\}
\]
from Lemma \ref{lem_good_I}, and the fact that $g_{\cI}(y) \leq 1$, we compute that
\[
\begin{split}
\sum_{y \in Y} \frac{g_{\cI}(y)}{y}F(y) 
&=  \sum_{y \in Y(\cI,\delta)} \frac{g_{\cI}(y)}{y}F(y) + \sum_{y \in Y\setminus Y(\cI,\delta)} \frac{g_{\cI}(y)}{y}F(y)
\\&\leq \frac{F(y_{\min})}{y_{\min}}\sum_{y \in Y(\cI,\delta)}1 + \delta\sum_{y \in Y} \frac{F(y)}{y}
\\&\leq \#Y(\cI,\delta) \frac{F(y_{\min})}{y_{\min}} + 3\delta,
\end{split}
\]
where we used \eqref{const_mass} in the last step. This shows that
\begin{equation}
\label{upper_bound_prop}\lambda\left(\bigcup_{q \in S}A_{q}\right) \leq 2\#Y(\cI,\varepsilon) \frac{F(y_{\min})}{y_{\min}} + 7\delta.
\end{equation}

Next, we define 
\[
G(y) = \frac{1}{BP} \sum_{z = y_{\min}x_{\min}+1}^{yx_{\max}} f(z),
\]
which inherits monotonicity from $f$.
By \eqref{control_xmin}, if $y \in Y$ then
\[
\begin{split} \lvert F(y) - G(y) \rvert &\le \frac{1}{BP} \sum_{z = y_{\min}x_{\min}+1}^{y_{\max}x_{\min}} f(z)
\leq \frac{1}{BP}y_{\max}x_{\min}f(x_{\min}'y_{\min})
\\&\leq 2y_{\max}f(x_{\min}'y_{\min}) = o_{Y}(1).
\end{split}
\]
Thus,
\begin{align*}\sum_{y \in Y} \sum_{x \in X} f(xy) &= \left(\sum_{y \in Y}\frac{F(y)}{y}\right) + o_{Y}(1) 
\\&=\left(\sum_{y \in Y}\frac{G(y) + o_{Y}(1)}{y}\right) + o_{Y}(1) 
\\&= \left(\sum_{y \in Y}\frac{G(y)}{y}\right) + o_{Y}(1).
\end{align*}
Combining this with \eqref{const_mass} and \eqref{upper_bound_prop} yields
\begin{equation}
\label{intermed_step}
\lambda\left(\bigcup_{q \in S}A_{q}\right) \leq 4\frac{\#Y(\cI,\varepsilon) \frac{F(y_{\min})}{y_{\min}}}{\sum_{y \in Y}\frac{G(y)}{y}} + 8\delta.
\end{equation}

Observe that
\[
\begin{split}\sum_{y \in Y}\frac{G(y)}{y} \geq G(y_{\min}) \sum_{y \in Y}\frac{1}{y}
&\geq F(y_{\min}) \sum_{y \in Y}\frac{1}{y}
\end{split}.
\]
Hence, by Lemma \ref{lem_good_I},
\[
\frac{\#Y(\cI,\varepsilon) \frac{F(y_{\min})}{y_{\min}}}{\sum_{y \in Y}\frac{G(y)}{y}}
\leq \frac{\#Y(\cI,\varepsilon)}{\sum\limits_{y \in Y}\frac{y_{\min}}{y}} < \delta.
\]
Substituting this into \eqref{intermed_step} and choosing $\del = \eps/12$ completes the proof of Proposition \ref{counterexample_blocks},
thereby completing the proof of Theorem \ref{thmA}.

\subsection{Proof of Theorem \ref{thmC}}
\label{sec:upperdensityproof}

Before proving Theorem~\ref{thmC}, we briefly justify the first part of Remark~\ref{rem:but} namely that, in Theorem \ref{thmA}, the function $\psi$ cannot necessarily be taken as $\bone_S f$ for some $S\subset\NN$ with full upper density. Suppose the function $f$ satisfies $f(q) \gg \frac{1}{\varphi(q)}$. Suppose
$\overline\delta(S) = 1$, and let $\eps \in (0, 1/2)$. Let $Q \ge 1$ be such that
\begin{equation*}
\frac{\#([1,Q]\cap\NN)\setminus S}{Q} < \eps. 
\end{equation*}
Then
\begin{align*}
  \sum_{q\in S} \frac{\varphi(q)f(q)}{q}
  &\geq \sum_{q\in S\cap [1,Q]} \frac{\varphi(q)f(q)}{q} \gg \sum_{q\in S\cap [1,Q]} \frac{1}{q}\geq \sum_{\eps Q <q \leq Q} \frac{1}{q} \gg \log(\eps^{-1}). 
\end{align*}
Since $\eps>0$ can be arbitrarily small, this shows that
$\sum_{q\in S} \frac{\varphi(q)f(q)}{q}=\infty$. By Theorem \ref{KMthm}, 
\[
\lambda(W(f\bone_S))=1.
\]

\

\begin{proof}[Proof of Theorem~\ref{thmC}]
The function $\psi:\NN\to\RR_+$ will be supported on pairwise disjoint blocks:
\begin{equation}
\label{BlockSupport}
\supp{\psi} = \bigcup_{j=1}^\infty (C_j \cup D_j).
\end{equation}
The blocks $C_j$ will be sparse, and defined in a way that contributes substantially to $\sum \psi(q)$ without significantly affecting $W(\psi)$. The blocks $D_j$ will be long strings of consecutive integers on which $\psi$ contributes negligibly to $\sum \psi(q)$. These serve to ensure that $\supp{S}$ has full upper density. 

To start the construction, put
$D_0 = \setdelim{1}$, $Q_0=1$, and $\psi(1) = f(1)$. For each $j \geq 1$, we will have $\psi(Q_{j-1}) > 0$, and we choose $M_j > Q_{j-1}$ large enough that $f(M_j) \leq \psi(Q_{j-1})$. By Proposition~\ref{counterexample_blocks}, we may find a finite set
$C_j \subset \NN \cap (M_j,\infty)$ such that
\begin{equation}
\label{eq:fromprop}
\sum_{q \in C_j} f(q)
\in [1/2,1] \quad\textrm{and}\quad \lambda \left(\bigcup_{q \in C_j} A_q \right) 
< 2^{-j}.
\end{equation}
Let $D_j = (\max C_j, Q_j]\cap \NN$, where $Q_j$ is chosen large enough that 
\begin{equation}
\label{eq:density}
\frac{Q_j - \max C_j}{Q_j} \geq 1 - \frac{1}{j}. 
\end{equation}
Define
\begin{equation*}
\label{eq:psi}
\psi(q) =
\begin{cases}
f(q) &\textrm{if } q \in C_j, \textrm{ for some } j\geq 1 \\
\min\setdelim{f(q), q^{-2}}, &\textrm{if } q \in D_j \textrm{ for some } j\geq 1\\
0, &\textrm{otherwise.}
\end{cases}
\end{equation*}

Clearly, we have $\psi(q)\leq f(q)$ for all $q\in\NN$, as well as \eqref{BlockSupport}. Notice that
$\supp\psi$ has upper density $1$ by~(\ref{eq:density}), and that $\psi$ is non-increasing on $\supp\psi$ by construction. Moreover, by~(\ref{eq:fromprop}),
\begin{align*}
\sum_{q=1}^\infty \psi(q)
\geq \sum_{j=1}^\infty \sum_{q\in C_j}\psi(q) 
= \sum_{j=1}^\infty \sum_{q\in C_j}f(q) 
= \infty.
\end{align*}

It remains to show that $\lam(W(\psi)) = 0$. To see this, note from (\ref{eq:fromprop})
and the first Borel--Cantelli lemma that almost no $x$ lie in infinitely many of the sets $\bigcup_{q \in C_j} A_q$. Furthermore, we have
\begin{equation*}
\sum_{j=1}^\infty\sum_{q\in D_j} \psi(q) \leq \sum_{q=1}^\infty q^{-2} < \infty, 
\end{equation*}
so almost no $x$ lie in infinitely many of the sets $A_q\, (q\in D_j)$. Therefore
\[
\lambda(W(\psi))=0.
\]
\end{proof}

\begin{rem}
The proof above can be straightforwardly adapted to the inhomogeneous setup, giving inhomogeneous constructions with upper density $1$ that are non-increasing on the support.
We omit a full description for the sake of conciseness; one only needs to replace the ``counterexample blocks'' $C_i$ by the inhomogeneous counterexample blocks from \S \ref{sec:inhom_c}. 
\end{rem}

\section{Inhomogeneous counterexamples}\label{sec:inhom_c}

In this section, we prove the inhomogeneous Theorems~\ref{thmB} and~\ref{thm:howliouville}. They are based on the following inhomogeneous version of Proposition~\ref{counterexample_blocks}.

\begin{prop}
\label{prop:inhomprop}
Let $f: \bN \to [0,1/2)$ be a non-increasing function satisfying \eqref{fdiverge}. Then, for any $\eps, M > 0$ and any positive integer $b$, there exists a finite set $S\subset \NN \cap (M,\infty)$ such that
\[
\sum_{q\in S} f(q) \in [1/2,1],
\]
and such that
\begin{equation}
\label{eq:smallmass}
\lambda\parens*{\bigcup_{q\in S} A_q^\gamma(f)} < \eps 
\end{equation}
holds for every
$\gamma \in \bigcup_{j\in\ZZ} \brackets*{\frac{j-\delta}{b},
   \frac{j+\delta}{b}}$, where $\delta = b f(\max S)$.
\end{prop}

The proof of Proposition~\ref{prop:inhomprop} relies on the following
observations. 

\begin{lem}
\label{lem:simple}
Let $\eps>0$, $\gamma\in\RR$, and $b\in\NN$. If $\norm{b\gamma} \leq b\eps$, then $A_q^\gamma(\eps) \subseteq A_{bq} (2 b\eps)$.
\end{lem}

\begin{proof}
Let $\alpha \in A_q^\gamma(\eps)$, so that there exists $a\in \ZZ$ with
\begin{equation*}
\abs*{\alpha - \frac{a+\gamma}{q}} <\frac{\eps}{q}.
\end{equation*}
For each $j \in \bZ$, the triangle inequality gives
\begin{align*}
\abs*{\alpha - \frac{a+\gamma}{q}} 
&= \abs*{\alpha - \frac{ba+ b\gamma}{bq}}
= \abs*{\alpha - \frac{ba+ j +(b\gamma-j)}{bq}} \\ &\geq \abs*{\alpha - \frac{ba+ j}{bq}} -  \abs*{\frac{b\gamma-j}{bq}}.
\end{align*}
Combining yields
\begin{equation*}
\abs*{\alpha - \frac{ba+ j}{bq}}< \frac{\eps}{q} +  \abs*{\frac{b\gamma-j}{bq}} \qquad (j \in \bZ).
\end{equation*}
Meanwhile, since $\norm{b\gamma} \leq b\eps$, there is some integer $j$
  such that $\abs{b\gamma - j}\leq b\eps$, so
      \begin{equation*}
        \abs*{\alpha - \frac{ba+ j}{bq}}< \frac{2\eps}{q} = \frac{2b\eps}{bq}.
      \end{equation*}
      This shows that $\alpha \in A_{bq}(2b\eps)$, hence
      $A_q^\gamma(\eps) \subseteq A_{b q}(2 b\eps)$ as claimed.
    \end{proof}
    
\begin{lem}
\label{lem:dilation}
Let $I_1, \dots, I_k$ be intervals, and $b \geq 1 $. Then
\begin{equation*}
\lambda\parens*{\bigcup_{i=1}^k b\bullet I_i} \leq b    \lambda\parens*{\bigcup_{i=1}^k I_i},
\end{equation*}
where $b\bullet I_i$ denotes the concentric dilation of $I_i$ by the
factor $b$.
\end{lem}

\begin{proof}
The intervals $I_1, \dots, I_k$ may be assumed without loss of generality to be open. Write
\begin{equation*}
\bigcup_{i=1}^k I_i =     \bigcup_{j=1}^\ell J_j, 
\end{equation*}
where $J_1, \dots, J_\ell$ are pairwise disjoint open intervals. Notice that
\begin{equation*}
\bigcup_{i=1}^k b\bullet I_i \subset     \bigcup_{j=1}^\ell b\bullet J_j. 
\end{equation*}
It follows that
\begin{align*}
\lambda\parens*{\bigcup_{i=1}^k b\bullet I_i}
\leq \lambda\parens*{\bigcup_{j=1}^\ell b\bullet J_j} \leq b \sum_{j=1}^\ell \lambda(J_j) = b \lambda\parens*{\bigcup_{j=1}^\ell J_j} = b \lambda\parens*{\bigcup_{i=1}^k I_i}.
\end{align*}
\end{proof}

\begin{proof}[Proof of Proposition~\ref{prop:inhomprop}]
By an application of Proposition~\ref{counterexample_blocks}, we may find a finite set $S\subset\NN\cap(M,\infty)$ such that
\begin{equation*}
\sum_{q\in S} 2 f (q) \in [1,2]
\qquad \textrm{and} \qquad
\lambda \parens*{\bigcup_{q\in S} A_q(2 f)} < \frac{\eps}{b}. 
\end{equation*}
Thus, by Lemma~\ref{lem:dilation},
\begin{equation*}
\sum_{q\in S}  f (q) \in [1/2,1]
\qquad \textrm{and} \qquad
\lambda\parens*{\bigcup_{q\in S} A_q(2 b f)} < \eps.
\end{equation*}
Put $\delta = b f(q_{\max})$, where $q_{\max} = \max S$, and let
\begin{equation*}
\label{eq:gammaints}
\gamma \in \bigcup_{j\in\ZZ} \brackets*{\frac{j-\delta}{b},
\frac{j+\delta}{b}}.
\end{equation*}
Then 
\begin{equation}
\label{eq:gammaapprox}
\norm{b\gamma} \leq \delta = b f(q_{\max})\leq b f(q),
\end{equation}
for every $q\in S$ so,  by Lemma~\ref{lem:simple}, 
\begin{equation*}
A_q^\gamma(f(q)) \subseteq A_{b q}(2 b f (q)).
\end{equation*}
Note that $f(q)$ is a number here, not the name of a function, so $A_{bq}(2bf(q))$ does not mean $A_{bq}(2bf) = A_{bq}(2bf(bq))$.
  
Finally, since $x \mapsto bx \pmod{1}$ is a measure-preserving transformation on $[0,1)$,
\begin{equation*}
\lambda\parens*{\bigcup_{q\in S}A_q^\gamma(f)} \le    \lambda\parens*{\bigcup_{q\in S}A_{b q}(2bf(q))} = \lambda\parens*{\bigcup_{q\in S}A_{q}(2 b f)} < \eps.
\end{equation*}
This establishes~(\ref{eq:smallmass}) and finishes the proof of the proposition.
\end{proof}

We are ready to prove Theorem~\ref{thmB}. The idea is to apply Proposition~\ref{prop:inhomprop} repeatedly with an increasing sequence of values for $b$. We will choose this sequence $(b_k)_{k \ge 1}$ so that if $q \in \bN$ then $q \mid b_k$ for all sufficiently large $k$. 

\begin{proof}
[Proof of Theorem~\ref{thmB}]
For $k\geq 1$, let $b_k = (p_1 p_2 \cdots p_k)^{G_k}$, where $(G_k)_{k \ge 1}$ is an
sequence of positive integers to be chosen later. For each $k \geq 1$, Proposition~\ref{prop:inhomprop} furnishes $\delta_k>0$ and a finite set
\begin{equation*}
S_k\subset \NN\cap (\max S_{k-1}, \infty)
\end{equation*}
such that
\begin{equation}
\label{eq:onceagain}
\sum_{q\in S_k} f(q) \in [1/2,1]
\qquad\textrm{and}\qquad
\lambda\parens*{\bigcup_{q\in S_k} A_q^\gamma(f)} < 2^{-k}
\end{equation}
for every
$\gamma \in \bigcup_{j\in\ZZ} \brackets*{\frac{j-\delta_k}{b_k},
\frac{j+\delta_k}{b_k}}$. Choose $G_k>G_{k-1}$ large enough that for each for each $k$, each interval from
$\bigcup_{j\in\ZZ} \brackets*{\frac{j-\delta_{k-1}}{b_{k-1}},
\frac{j+\delta_{k-1}}{b_{k-1}}}$ contains at least two intervals from
$\bigcup_{j\in\ZZ} \brackets*{\frac{j-\delta_{k}}{b_k},
\frac{j+\delta_{k}}{b_k}}$. Let $S = \bigcup_{k\geq 1} S_k$, and note from (\ref{eq:onceagain}) that
\begin{equation*}
\sum_{q\in S} f(q) = \sum_{k=1}^\infty \sum_{q\in S_k}f(q) = \infty.
\end{equation*}

Put
\begin{equation*}
\Gamma = 
\liminf_{k\to\infty} \bigcup_{j\in\ZZ} \brackets*{\frac{j-\delta_{k}}{b_k},
\frac{j+\delta_{k}}{b_k}}, 
\end{equation*}
and notice that $\Gamma$ is uncountable and that $\QQ \subseteq \Gamma$. After all, if $a \in \bZ$ and $q \in \bN$ then, for $k$ sufficiently large,
\begin{equation*}
\frac{p}{q}\in   \bigcup_{j\in\ZZ} \brackets*{\frac{j-\delta_{k}}{b_k},
\frac{j+\delta_{k}}{b_k}}. 
\end{equation*}

Let $\gamma \in \Gamma$.  By~(\ref{eq:onceagain}) and the fact that
$\sum 2^{-k}<\infty$, we have
\begin{equation*}
\lambda\parens*{\limsup_{k\to\infty} \bigcup_{q\in S_k}A_q^\gamma} = 0,
\end{equation*}
whence $\lambda(W(\bone_S f,\gamma))=0$.
\end{proof}

The condition~\eqref{eq:gammaapprox} quantifies the diophantine approximation rate of
$\gamma$ in terms of $f$. Informally, Theorem~\ref{thmB} may be read
as saying that for every $f$ there is a set $S \subseteq \NN$, and a set
$\Gamma \subseteq \RR$ consisting of real numbers $\gamma$ which are
``wildly Liouville in terms of $f$'', such that every $\gamma \in \Gamma$
is a counterexample for the function $\bone_S f$.

In the next theorem we specify a notion of ``sufficiently Liouville in terms of $f$'', and show that for \emph{every} $\gamma$ meeting this
condition there exists $S=S_\gamma \subseteq \NN$ such that $\gamma$ is
a counterexample for $\bone_S f$.

\begin{thm}
\label{thm:generalhowliouville}
Let $f:\NN \to [0,\infty)$ be a non-increasing function satisfying \eqref{fdiverge}. Then there exists an increasing, unbounded function $L: \bN \to [0,\infty)$ associated to $f$ such that if
\begin{equation}
\label{eq:generalhowliouville}
\liminf_{q \to \infty} L (q) \lVert q \gamma \rVert < \infty,
\end{equation}
then there exists $S = S_\gamma \subseteq \N$ such that we have \eqref{eq:thmB}.
\end{thm}

\begin{proof}
Let $H: (0,\infty) \to \NN$ be an increasing, unbounded function such that in Proposition~\ref{counterexample_blocks} for $M = 1$ and any $\varepsilon > 0$, one can find such a set $S$ with $\max S\leq H(1/\eps)$. We choose
\begin{equation} \label{LH}
L(q) = \frac{1}{f(H(q^2))}.
\end{equation}

Suppose $\gamma\in\RR$ satisfies~(\ref{eq:generalhowliouville}). Then there
exists an increasing sequence $(q_n)_{n \ge 1}$ of positive integers such that
\begin{equation}
\label{eq:epsilonmod}
L(q_n) \norm{q_n\gamma} \leq q_n \qquad (n\geq 1)
\end{equation}
and $\sum q_n^{-1} < \infty$. For each $n \in \bN$, Proposition~\ref{prop:inhomprop} provides a set
$S_n \subset \NN \cap (1, \infty)$ with
\begin{equation}
\label{eq:measureHf}
\sum_{q\in S_n} f(q) \in [1/2,1]
\quad\textrm{and}\quad\lambda\parens*{\bigcup_{q\in S_n}A_q^\gamma(f)} < \frac{1}{q_n}
\end{equation}
holding for $\gamma$ as long as
\begin{equation}
\label{eq:whatweneed}
\norm{q_n \gamma} \leq q_n f(\max S_n).
\end{equation}

Next, we find an upper bound for $\max S_n$, by working through the proof
of Proposition~\ref{prop:inhomprop}. Applying Proposition \ref{counterexample_blocks} with $\varepsilon = q_n^{-2}$ and $M = 1$, we get by the very definition of $H$ that $\max S_n \le H(q_n^2)$. In order for the proof of Proposition~\ref{prop:inhomprop} to carry through with parameters $\varepsilon = q_n^{-1}$ and $b = q_n$, we need to ensure that condition \eqref{eq:gammaapprox} is met, so it suffices to have
\begin{equation*}
\norm{q_n\gamma} \leq q_n f(H(q_n^2)),
\end{equation*}
but this is given
by~(\ref{eq:epsilonmod}). The upshot is that (\ref{eq:measureHf}) holds for all $n \in \bN$. 
  
By choosing $(q_n)_{n \ge 1}$ to grow sufficiently quickly that 
\[
q_n^{-1} < \lambda(A_{\max S_{n-1}}^\gamma(f)) \qquad (n \ge 2),
\]
we ensure that $\max S_{n-1} < \min S_n$. Putting $S = \bigcup S_n$, we have $\sum_{q \in S} f(q) = \infty$ and, since
$\sum q_n^{-1} < \infty$, we have $\lambda(W(\bone_S f, \gamma)) = 0$.
\end{proof}

In the special case of 
functions of the form
\eqref{special}, we are able to specify a function $H$ as above, and then use this to prove Theorem~\ref{thm:howliouville}.

\begin{lem}
\label{lem:maxS}
In Proposition~\ref{counterexample_blocks}, for $f_k$, one may find $S$ with
\begin{equation*}
\max S \leq \exp^{(k+3)} \parens*{\eps^{-3000}},
\end{equation*}
as long as $\eps >0$ is small enough. 
\end{lem}

\begin{proof}
In the proof of Lemma~\ref{lem_good_I},  for $\delta > 0$ sufficiently small, we take $\cI = \cI_K$, and here we can choose $K > \del^{-2000}$ minimally.
Denoting \eqref{denote},  a standard estimate for primorials gives
\begin{align*}
P \leq \prod_{p \leq e^{K+1}}p = e^{(1+o(1))e^{K+1}} \leq \exp^{(2)} \parens*{\del^{-2001}}.
\end{align*} 
The proof of Proposition~\ref{counterexample_blocks} concludes with selecting $\delta = \eps/12$. Taking this linear relationship into
account, we see that we may find $Y$ such that
\begin{equation*}
\max Y \leq P \leq \exp^{(2)} \parens*{\eps^{-2025}},
\end{equation*}
as long as $\eps$ is sufficiently small.

We now revisit Steps~\eqref{step1}, \eqref{step2}, and \eqref{step3} of the construction of $S$ in the proof of Proposition~\ref{counterexample_blocks}. In Step~\eqref{step1}, we choose $x_{\min}'$ large enough that $f(x_{\min}'y_{\min})$ is small, even after multiplying by a
parameter depending on $Y$. Noting that $P \ge y_{\max} \ge |Y|$, the proof reveals that
\begin{equation*}
2P f(x_{\min}'y_{\min}) \leq \del/2
\end{equation*}
suffices. Since we have $f=f_k$, we may choose $x_{\min}' = 4P/\del$.
    
In Step~\eqref{step2}, we choose $B \in \bN$ with $(B,P)=1$, such that 
\[
x_{\min} := BP + 1 
\ge x'_{\min}.
\]
With $f=f_k$ and $x_{\min}'$ as
above, we can make the specific choice $B = P - 1$.

In Step~\eqref{step3}, we let $x_{\max}$ be the smallest integer such that
\begin{equation*}
\sum_{y\in Y}\sum_{\substack{x\in [x_{\min}, x_{\max}] \\ x \equiv 1 \pmod{BP}}} f_k(xy) \in [1,2]. 
\end{equation*}
We compute that
\begin{align*}
\sum_{y\in Y} \sum_{\substack{x \in [x_{\min}, x_{\max}] \\ x \equiv 1 \pmod{BP}}} f_k(xy)
&= \sum_{y\in Y}\sum_{\substack{x\in [x_{\min}, x_{\max}] \\ 
x \equiv 1 \pmod{BP}}} 
\frac{1}{xy\log(xy)\cdots\log^{(k)} (xy)} \\
&\geq \frac{1}{2^k}  \parens*{\sum_y\frac{1}{y}}
\sum_{\substack{x\in [x_{\min}, x_{\max}] \\ x \equiv 1 \pmod{BP}}} \frac{1}{x\log(x) \cdots \log^{(k)} (x)} \\
&\geq \frac{1}{4^kBP} \parens*{\sum_y\frac{1}{y}} \parens*{\log^{(k+1)}x_{\max} - \log^{(k+1)}x_{\min}} \\
&\geq \frac{1}{4^k P^3} \parens*{\log^{(k+1)}x_{\max} - \log^{(k+1)}(P^2)},
\end{align*}
since $\sum_y y^{-1} \geq 1/P$. We therefore have
\begin{equation*}
\frac{1}{4^k P^3} \parens*{\log^{(k+1)}x_{\max} - \log^{(k+1)}(P^2)} \leq 2,
\end{equation*}
hence
\begin{align*}
\log^{(k+1)}x_{\max} \leq 4^{k+1} P^3 + \log^{(k+1)}(P^2) \leq \exp^{(2)} \parens*{\eps^{-2500}},
\end{align*}
for $\eps$ sufficiently small, and finally
\[
S_{\max} = x_{\max} y_{\max} \le P x_{\max} \le \exp^{(k+3)}(\eps^{-3000}).
\]
\end{proof}

\begin{proof}[Proof of Theorem~\ref{thm:howliouville}]
We deploy the notation \eqref{special},  and let $H: (0,\infty) \to \bN$ be an increasing, unbounded function such that in Proposition~\ref{counterexample_blocks} we may find $S$ satisfying $\max S \leq H(1/\eps)$. By Lemma~\ref{lem:maxS}, there is such a function $H$ satisfying
\begin{equation*}
H(x) \leq \exp^{(k+3)}(x^{3000})
\end{equation*}
for all sufficiently large $x$.

In view of Theorem~\ref{thm:generalhowliouville} and the choice \eqref{LH} from its proof, for every
$\gamma\in\RR$ such that
\begin{equation*}
\liminf_{q \to \infty} \frac{\lVert q \gamma \rVert }{f_k(H(q^2))}  < \infty,
\end{equation*}
there exists $S = S_\gamma \subseteq \N$ such that
\begin{equation*}
\sum_{q \in S}
f_k(q) = \infty, \qquad \lambda(W(\bone_S f_k,\gamma)) 
= 0.
\end{equation*}
But
\begin{equation*}
\frac{1}{f_k(H(q^2))}  \leq    \frac{1}{f_k(\exp^{(k+3)}(q^{6000}))}  < \exp^{(k+3)} (q^{7000})
\end{equation*}
for all sufficiently large $q$, therefore
\begin{equation*}
\liminf_{q\to\infty} \norm{q\gamma}\exp^{(k+3)}(q^{7000}) < \infty
\end{equation*}
suffices and the theorem is proved.
\end{proof}


\begin{thebibliography}{99}

\bibitem{AR}
D. Allen and F. Ram\'irez, \emph{Independence inheritance and Diophantine approximation for systems of linear forms}, Int. Math. Res. Not. \textbf{2023,} 1760--1794. 

\bibitem{BHV}
V. Beresnevich, M. Hauke and S. Velani, \emph{Borel-Cantelli, zero-one laws and inhomogeneous Duffin-Schaeffer}, arXiv:2406.19198.

\bibitem{BRV2016}
V. Beresnevich, F. Ram\'irez and S. Velani, \emph{Metric Diophantine Approximation: some
aspects of recent work}, Dynamics and Analytic Number Theory, London Math. Soc. Lecture Note Ser. (N.S.) \textbf{437,} Cambridge University Press, 2016, pages 1--95.

\bibitem{CT2024}
S. Chow and N. Technau, \emph{Littlewood and Duffin--Schaeffer-type problems in diophantine approximation}, Mem. Amer. Math. Soc. \textbf{296} (2024), 74 pages.

\bibitem{DS1941}
R. J. Duffin and A. C. Schaeffer, Khintchine's problem in metric Diophantine approximation, Duke Math. J. \textbf{8} (1941), 243--255.

\bibitem{Erd1970}
P. Erd\H{o}s, \emph{On the distribution of convergents of almost all real numbers}, J. Number Theory \textbf{2} (1970), 425--441.

\bibitem{H25}
M. Hauke, \emph{Quantitative inhomogeneous Diophantine approximation for systems of linear forms}, Proc. Amer. Math. Soc. \textbf{53} (2025), 1867--1880.

\bibitem{HR}
M. Hauke and F. Ram\'irez, \emph{The Duffin-Schaeffer conjecture with a moving target}, arXiv:2407.05344.

\bibitem{K25}
S. Kim, \emph{Inhomogeneous Khintchine-Groshev theorem without monotonicity}, arXiv:2411.07932.

\bibitem{KM2020}
D. Koukoulopoulos and J. Maynard, \emph{On the Duffin-Schaeffer conjecture}, Ann. of Math. \textbf{192} (2020), 251--307.

\bibitem{PV1990}
A. D. Pollington and R. C. Vaughan, \emph{The $k$-dimensional Duffin and
Schaeffer conjecture}, Mathematika \textbf{37} (1990), 190--200.

\bibitem{Ram2017}
F. Ram\'irez, \emph{Counterexamples, covering systems, and zero-one laws for inhomogeneous approximation}, Int. J. Number Theory \textbf{13} (2017), 633--654.

\bibitem{Vaa1978}
J. D. Vaaler, \emph{On the metric theory of Diophantine approximation}, Pacific J. Math. \textbf{76} (1978), 527--539.

\bibitem{Yu2021}
H. Yu, \emph{On the metric theory of inhomogeneous Diophantine approximation: An Erd\H{o}s--Vaaler type result}, J. Number Theory
\textbf{224} (2021), 243--273.

\end{thebibliography}
\end{document}